\documentclass[12pt,a4paper]{amsart}
\usepackage{amsfonts}
\usepackage{amsthm}
\usepackage{amsmath}
\usepackage{amscd}
\usepackage[latin2]{inputenc}
\usepackage{t1enc}
\usepackage[mathscr]{eucal}
\usepackage{indentfirst}
\usepackage{graphicx}
\usepackage{graphics}
\usepackage{pict2e}
\usepackage{epic}
\numberwithin{equation}{section}
\usepackage[margin=2.9cm]{geometry}
\usepackage{epstopdf} 
\usepackage{enumerate}
\usepackage{verbatim}
\usepackage{cite}
\usepackage{mathtools}
\usepackage{mathrsfs}
\usepackage[dvipsnames]{xcolor}
\usepackage{hyperref}
\usepackage{relsize}
\usepackage{esint}
\usepackage{comment}

\theoremstyle{plain}
\newtheorem{Th}{Theorem}[section]
\newtheorem{Lemma}[Th]{Lemma}
\newtheorem{Cor}[Th]{Corollary}
\newtheorem{Prop}[Th]{Proposition}
\newtheorem{Claim}[Th]{Claim}

\newtheoremstyle{named}{}{}{\itshape}{}{\bfseries}{.}{.5em}{\thmnote{#3}}
\theoremstyle{named}

\theoremstyle{definition}
\newtheorem{Def}[Th]{Definition}

\newtheorem{Rem}[Th]{Remark}

\newcommand{\area}{{\rm{area}}}
\newcommand{\vol}{{\rm{vol}}}

\definecolor{ao(english)}{rgb}{0.0, 0.5, 0.0}

\begin{document}

\title[Area bounds for CMC surfaces in hyperbolic 3-manifolds]{Area bounds for constant mean curvature surfaces in hyperbolic 3-manifolds}

\author{Ruojing Jiang}

\address{Massachusetts Institute of Technology, Department of Mathematics, Cambridge, MA 02139} 
\email{ruojingj@mit.edu}

 \subjclass[2020]{} 
 \date{}

\begin{abstract}  
On a finite-volume hyperbolic $3$-manifold, we establish an upper bound on the area of closed embedded surfaces with constant mean curvature at least one, depending on the mean curvature and the genus bounds. This area bound implies compactness for such surfaces. In particular, for Bryant surfaces with constant mean curvature equal to one, the area is bounded proportionally to the genus.
\end{abstract}

\maketitle

\section{introduction}
The theory of minimal surfaces provides a central motivation for the study of constant mean curvature (CMC) surfaces in $3$-manifolds, highlighting analogies and comparisons with minimal surfaces. In particular, essential (or incompressible) CMC surfaces are known not to exist in hyperbolic $3$-manifolds when the mean curvature satisfies $H\geq 1$, which prevents the direct use of surface subgroup methods that are standard in the minimal surface setting. Most known constructions of CMC surfaces in closed $3$-manifolds instead arise as boundaries of Caccioppoli sets, and are therefore completely compressible. A landmark result in this direction is the work of Zhou and Zhu \cite{Zhou-Zhu}, who used min-max theory to prove that for any closed manifold $M$ of dimension $3\leq n+1\leq 7$ and any $H\in\mathbb{R}$, there exists a closed, almost embedded hypersurface in $M$ with constant mean curvature $H$. An immersed hypersurface is \emph{almost embedded} if it locally decomposes into smoothly embedded components that pairwise lie on one side of each other. This condition is essential, since CMC hypersurfaces satisfy only a one-sided maximum principle. Subsequent work has clarified and extended the existence and regularity theory in various directions~\cite{ZhouZhuzPMC,Bellettini-Wickramasekera,Sarnataro-Stryker}.

Beyond existence and regularity, a natural next step is to investigate area bounds and compactness for closed CMC surfaces. This direction is inspired in part by the theory of \emph{Bryant surfaces}---surfaces in hyperbolic $3$-space with constant mean curvature equal to one---which are isometric to minimal surfaces in Euclidean $3$-space via a holomorphic parameterization analogous to the Weierstrass-Enneper representation~\cite{Bryant}. Due to this fact, a closed surface $S$ of a hyperbolic $3$-manifold $M$ obtained by projection of a Bryant surface (also referred to as a Bryant surface in this paper) may share some analogies with a closed minimal surface $\Sigma$ in a flat $3$-torus. This correspondence allows one to compare the properties of $S$ with those of $\Sigma$. For example, in a flat $3$-torus $\mathbb{T}^3$, for any $g\geq 3$, there exists an infinite sequence of embedded closed minimal surfaces of genus $g$ with area diverging to infinity. This argument was conjectured by Meeks \cite{Meeks}, who proved the genus $3$ case using the min-max method, and ultimately solved by Traizet \cite{Traizet}. These observations naturally raise the question: do Bryant surfaces---or more generally, CMC surfaces in hyperbolic manifolds---also exhibit unbounded area growth?

Our paper addresses this question from two perspectives. First, building on the compactness theorem of Meeks and Tinaglia~\cite{Meeks-Tinaglia}, which establishes uniform area bounds for closed $H$-surfaces in $\mathbb{T}^3$ with $0 <h_0\leq |H|\leq H_0$ and bounded genus, we investigate whether similar compactness and area bounds hold for closed $H$-surfaces in finite-volume hyperbolic $3$-manifolds, focusing on the case $|H| \geq 1$.

In the second part, we focus specifically on closed Bryant surfaces in closed hyperbolic $3$-manifolds. We show that such surfaces must have genus at least $3$ and are completely compressible. Using this structure, we derive explicit area bounds in terms of the genus, and extend our analysis to the case of closed Bryant surfaces in finite-volume hyperbolic manifolds.

\subsection{Main results}
Let $M$ be a hyperbolic 3-manifold of finite volume. We prove the following area bound for closed embedded CMC surfaces in $M$. This is an analogue of the case of $H$-surfaces in $\mathbb{T}^3$ with $0<|H|\leq H_0$, studied by Meeks and Tinaglia \cite{Meeks-Tinaglia}.

\begin{Th}\label{thm_area_bound}
Let $M$ be a finite-volume hyperbolic 3-manifold. Given $H_0\geq 1$ and $g_0\geq 0$, there exists a constant $C(M,H_0,g_0)>0$ that depends only on $M,H_0$ and $g_0$, such that any closed $H$-surface $S$ embedded in $M$ with $1\leq |H|\leq H_0$ and genus at most $g_0$ satisfies 
    \begin{equation*}
        \area(S)\leq C(M, H_0, g_0).
    \end{equation*}
\end{Th}

The area bound implies the compactness of closed embedded CMC surfaces. To state the compactness theorem, we need the following definition. 

\begin{Def}
    An immersed surface $S\subset M$ is said to be \emph{strongly Alexandrov embedded} if the immersion extends to an injective immersion in a domain $B\subset M$ with $\partial B=S$.
\end{Def}

\begin{Th}\label{thm_compactness}
    Let $M$ be a finite-volume hyperbolic 3-manifold, and let $\{S_n\}_{n\in\mathbb{N}}$ be a sequence of closed embedded $H_n$-surfaces in $M$ with $1\leq |H_n|\leq H_0$ and genus at most $g_0\geq 0$. After passing to a subsequence, $S_n$ converges smoothly to a strongly Alexandrov embedded $H$-surface $S$ (possibly disconnected) away from a finite set of points $\mathcal{Q}$, which lie in the self-intersection set of $S$. The convergence to each component of $S$ occurs with multiplicity one. Additionally, $S$ has constant mean curvature satisfying $1\leq |H|\leq H_0$, and its genus is at most $g_0$.
\end{Th}

Furthermore, for Bryant surfaces, we establish an area bound that depends explicitly on the genus.

\begin{Th}\label{thm_bryant area bound}
    Let $M$ be a closed hyperbolic 3-manifold.
    There exists a constant $C(M)>0$ depending only on $M$, such that for any closed Bryant surface $S$ immersed in $M$ with genus $g$, we obtain $g\geq 3$, and
    \begin{equation*}
        \text{area}(S)\leq C(M)g.
    \end{equation*}
\end{Th}

\begin{Cor}\label{cor_bryant area bound}
     Let $M$ be a finite-volume hyperbolic 3-manifold.
    There exists a constant $C(M)>0$ depending only on $M$, such that for any closed Bryant surface $S$ immersed in $M$ with genus $g\neq 1$, we obtain $g\geq 3$, and
    \begin{equation*}
        \text{area}(S)\leq C(M)g.
    \end{equation*}
\end{Cor}

\subsection{Organization}
The paper is organized as follows. Section~\ref{section_preliminary} reviews the background on CMC surfaces and Bryant surfaces. In Section~\ref{section_intrinsic_est}, we prove the intrinsic and extrinsic curvature estimates for CMC surfaces, which are used later in the proofs of the main theorems. Section~\ref{section_area_compactness} contains the proofs of Theorems~\ref{thm_area_bound} and~\ref{thm_compactness}, and Section~\ref{section_bryant} proves Theorem~\ref{thm_bryant area bound} and Corollary~\ref{cor_bryant area bound}.

\subsection*{Acknowledgements}
The author would like to thank Andr\'{e} Neves for helpful suggestions related to this work.

\section{Preliminaries}\label{section_preliminary}
\subsection{Constant mean curvature surfaces and Lawson correspondence}

Let $(M,g)$ be a Riemannian $3$-manifold. An immersed oriented surface $S\subset M$ with induced metric $g|_S$ and unit normal $\nu$ has second fundamental form $A$ and mean curvature $H = \frac{1}{2}\operatorname{tr}(A)$. A surface is called a \emph{constant mean curvature} (CMC) surface if $H$ is constant on $S$. The case $H=0$ corresponds to minimal surfaces.

A key insight due to Lawson \cite{Lawson1970} is that CMC surfaces in one space form can be transformed into CMC or minimal surfaces in another. More precisely, suppose that $S$ is a Riemann surface and $f:S \to M(c)$ is a conformal immersion into the simply connected space form of sectional curvature $c$. Then there is a correspondence (called the \emph{Lawson correspondence}) which takes a CMC $H$-immersion in $M(c)$ to a minimal immersion in a different space form $M(c+H^2)$, while preserving the conformal structure and Hopf differential.  

In hyperbolic space, an $H$-surface $S$ in $\mathbb{H}^3$ corresponds to an $(H-1)$-surface in $\mathbb{R}^3$ by the Lawson correspondence. More precisely, let $(I,A)$ denote the first and second fundamental forms of $S$. Then there exists an $(H-1)$-surface in $\mathbb{R}^3$, denoted by $\Sigma$, whose first and second fundamental forms are $(I,A-I)$. In particular, a CMC surface in $\mathbb{H}^3$ with constant mean curvature equal to one corresponds to a minimal surface in $\mathbb{R}^3$.

Thus, analytic data such as the induced conformal metric and the quadratic Hopf differential 
transform naturally under this correspondence. This provides a powerful bridge between the theories of minimal and constant mean curvature surfaces.

\subsection{Bryant surfaces in hyperbolic space}

Within hyperbolic $3$-space $\mathbb{H}^3$, surfaces of constant mean curvature $H=1$ 
are analogues of minimal surfaces in $\mathbb{R}^3$. Bryant \cite{Bryant} established that such surfaces admit a \emph{Weierstrass-type representation}, now known as the \emph{Bryant representation}. Surfaces arising in this way are called \emph{Bryant surfaces}. 

Concretely, let $S$ be a Riemann surface.  
Given a meromorphic function $g:S \to \mathbb{C} \cup \{\infty\}$ (called the \emph{hyperbolic Gauss map}) and a holomorphic $1$-form $\omega$ on $S$, one can construct a conformal immersion 
$f:S \to \mathbb{H}^3$ with constant mean curvature $H=1$. The induced metric and Hopf differential are expressed in terms of $(g,\omega)$ by
\[
ds^2 = \frac{(1+|g|^2)^2}{4}\,|\omega|^2, 
\qquad Q = \omega \, dg,
\]
where $Q$ is the Hopf differential of the immersion. The pair $(g,\omega)$ thus encodes all local data of the surface, which is analogous to the classical Weierstrass-Enneper representation of minimal surfaces in $\mathbb{R}^3$.

\section{Intrinsic curvature estimates}\label{section_intrinsic_est}
Suppose $M$ is a finite-volume hyperbolic 3-manifold. Our goal of this section is to establish the following intrinsic curvature estimate for CMC surfaces immersed in $M$.

\begin{Prop}\label{prop_curv_est_sep}
If $S$ is a closed $H$-surface immersed in $M$ with $|H|\geq 1$ and injectivity radius $I_S>0$, then there exists a constant $C=C(I_S)$ depending only on $I_S$, such that \begin{equation*}
    |A|_{L^\infty(S)}\leq C.
\end{equation*}

If $S$ is embedded, then it separates $M$. Moreover, there exist positive constants $\epsilon=\epsilon(I_S)$, $k=k(I_S)$, and a one-sided $\epsilon$-neighborhood $N_\epsilon(S)$ of $S$ in the mean convex component of $M\setminus S$, such that \begin{equation*}
    \area(S)\leq k\text{vol}(N_\epsilon(S)).
\end{equation*}
\end{Prop}

In particular, we will focus exclusively on the intrinsic curvature estimate for CMC disks embedded in $\mathbb{H}^3$ (Proposition~\ref{prop_intrinsic}) stated below. Once this is proved, the curvature estimate in Proposition~\ref{prop_curv_est_sep} follows directly as a consequence. 

Additionally, the separation property is derived from Lemma~\ref{lem_separating}. The construction of the one-sided $\epsilon$-neighborhood is outlined in Theorem 3.5 of \cite{Meeks-Tinaglia_epsilon-nbhd}, the constants $\epsilon$ and $k$ depend on the sectional curvature of the ambient manifold $M$, as well as the lower bound of $|H|$ and the upper bound of $|A|$ for the CMC surface $S$. Since the sectional curvature of $M$ and the lower bound of $|H|$ are fixed numbers $-1$ and $1$, respectively, and the first part of the proposition shows that $|A|$ of $|S|$ is bounded by $C(I_S)$, the constants $\epsilon$ and $k$ ultimately depend only on $I_S$.

\begin{Prop}[Intrinsic curvature estimates]\label{prop_intrinsic}
    Given $\delta>0$, there exists a constant $C=C(\delta)$ such that the following holds. For any $|H|\geq 1$ and any embedded $H$-disk $D\subset \mathbb{H}^3$, we have \begin{equation*}
        \sup_{x\in D:\, d_D(x,\partial D)\geq \delta}|A|\leq C.
    \end{equation*}
\end{Prop}

We will follow the strategy presented in Section 5 of \cite{Tinaglia_review} and outline the key steps. In the following proofs, we assume $H\geq 1$, since the case where $H\leq -1$ can be derived in a similar manner. Additionally, we note that the proof remains valid if $|H|$ is bounded below by a positive constant.

\subsection{Separation property}

First, when the surface is embedded in $M$, we deduce the separation property.
\begin{Lemma}\label{lem_separating}
    Let $S$ be a closed $H$-surface immersed in $M$, where $|H|\geq 1$. In this case, $S$ is not an essential surface. Furthermore, if $S$ is embedded, it divides $M\setminus S$ into two connected components.
\end{Lemma}
\begin{proof}

    Let $\Tilde{S}$ be a lift of $S$ in $\mathbb{H}^3$. If $S$ were an essential surface in $M$, then $\Tilde{S}$ would be a disk, and its boundary at infinity, $\partial_\infty\Tilde{S}$, would form a simple closed curve in $S^2_\infty$. This scenario can be ruled out by comparing the mean curvature of $\Tilde{S}$ with that of the horospheres. 
    
    To be more specific, suppose $x\in S^{2}_\infty\setminus \partial\Tilde{S}$ is a point that lies in the mean convex side of $\Tilde{S}$, we recall that $\mathbb{H}^3$ is foliated by a family of horospheres \emph{centered} at $x$, that is, all these horospheres are asymptotic to $x$ on $S^{2}_\infty$. We denote this family by $\mathcal{S}_{x}$. 
    We can choose a horosphere $S_1\in\mathcal{S}_x$ such that $S_1$ is tangent to $\Tilde{S}$ at a point, and their mean curvature vectors are pointing in the same direction. This implies that the mean curvature of $\Tilde{S}$, given by $H$, must be strictly less than the mean curvature of the horosphere $S_1$, which is equal to $1$. However, this contradicts the assumption that $H\geq 1$. Therefore, $S$ cannot be an essential surface.

    Furthermore, suppose $S$ is embedded. Since $S$ is not essential, it does not represent a non-zero element in $H_2(M;\mathbb{Z})$. In other words, $S$ separates $M$.
\end{proof}

\subsection{Extrinsic curvature estimates}\label{subsection_extrinsic}
We prove Proposition~\ref{prop_intrinsic} in Sections~\ref{subsection_extrinsic} and~\ref{subsection_intrinsic}.

\begin{Lemma}[Weak extrinsic curvature estimate]\label{lemma_weak extrinsic}
    Let $k>0$ and $H>k$. Given small $\epsilon>0$, there exists a constant $C=C(\epsilon)$ such that the following holds. Let $D$ be an $H$-disk embedded in $\mathbb{H}^3(-k^2)$, and satisfy $0\in D$ and $d_{\mathbb{H}^3(-k^2)}(0,\partial D)\geq\frac{\epsilon}{k}\tanh^{-1}\left(\frac{k}{H}\right)$. Then \begin{equation*}
        |A|(0)\leq CH.
    \end{equation*}
\end{Lemma}

\begin{proof}
    Suppose by contradiction that for some small constant $\epsilon>0$, there exists a sequence of $1$-disks $D_n\subset \mathbb{H}^3(-k_n^2)$ that contain the origin and satisfy $d_{\mathbb{H}^3(-k_n^2)}(0,\partial D_n)\geq\frac{\epsilon}{k_n}\tanh^{-1}(k_n)$, where $k_n<1$. However, \begin{equation}\label{equ_assump_extrinsic}
        |A_{D_n}|(0)>n.
    \end{equation}
    
    Let $\Tilde{D}_n:=|A_{D_n}|(0)D_n$ be a rescaling of $D_n$ near the origin. Then it is a $\frac{1}{|A_{D_n}|(0)}$-disk in the simply-connected $3$-space with constant sectional curvature $-\frac{k_n^2}{|A_{D_n}|^2(0)}$. Since $|A_{D_n}|(0)>1$, we have  
    \begin{align}\label{equ_dist_lower}
       d_{\mathbb{H}^3\left(-\frac{k_n^2}{|A_{D_n}|^2(0)}\right)}(0,\partial \Tilde{D}_n)\geq & |A_{D_n}|(0)\cdot \frac{\epsilon}{k_n}\tanh^{-1}(k_n)=\frac12 \frac{\epsilon}{\frac{k_n}{|A_{D_n}|(0)}}\ln\left(\frac{1+k_n}{1-k_n}\right)\\\nonumber
       >& \frac12 \frac{\epsilon}{\frac{k_n}{|A_{D_n}|(0)}}\ln\left(\frac{1+\frac{k_n}{|A_{D_n}|(0)}}{1-\frac{k_n}{|A_{D_n}|(0)}}\right)=\frac{\epsilon}{\frac{k_n}{|A_{D_n}|(0)}}\tanh^{-1}\left(\frac{k_n}{|A_{D_n}|(0)}\right).
    \end{align} 
    Additionally, $\Tilde{D}_n$ has uniformly bounded norm of the second fundamental form in an extrinsic ball. Therefore, the standard compactness theorem says that, after passing to a subsequence, $\Tilde{D}_n$ converges smoothly on compact sets to a minimal lamination $\mathcal{L}$ of $\mathbb{R}^3$. 
    
    Since $|A_{\Tilde{D}_n}|(0)=1$, we obtain $|A_{\Tilde{D}_\infty}|(0)=1$, where $\Tilde{D}_\infty$ is the leaf of $\mathcal{L}$ that contains the origin. Hence, $\Tilde{D}_\infty$ is non-flat. Taking $n\to\infty$ in \eqref{equ_dist_lower}, we get $d_{\mathbb{R}^3}(0,\partial\Tilde{D}_\infty)\geq \epsilon>0$.
    Moreover, since $\Tilde{D}_\infty$ is a complete connected minimal surface in $\mathbb{R}^3$ with bounded second fundamental form, it must be proper (\cite{Meeks-Rosenberg_properness}, Theorem 2.1).
    According to the halfspace theorem, any connected proper non-planar minimal surface in $\mathbb{R}^3$ cannot be contained entirely in a halfspace (\cite{Hoffman-Meeks}, Theorem 1). When combined with Theorem 1.6 of \cite{Meeks-Rosenberg}, it follows that if $\mathcal{L}$ has more than one leaf, then the leaves must be planes. This implies that $\Tilde{D}_\infty$ is the only leaf in $\mathcal{L}$. 
    
    Furthermore, we observe that $\Tilde{D}_n$ converges to $\Tilde{D}_\infty$ with multiplicity at most two. If the multiplicity were at least three, then for sufficiently large $n$, we could identify disjoint components $\Tilde{D_n^1}$, $\Tilde{D_n^2}$ of $\Tilde{D}_n$, where each component is a graph over the limit $\Tilde{D}_\infty$. Additionally, assume that the inner product of the unit normal vectors of $\Tilde{D_n^1}$ and $\Tilde{D_n^2}$ converges to one as $n\rightarrow \infty$. Denote the graphs of $\Tilde{D_n^1}$, $\Tilde{D_n^2}$ by $f_n^1$, $f_n^2$, respectively. They satisfy $f_n^1-f_n^2>0$. Using the elliptic PDE theory and normalizing $f_n^1-f_n^2$ at some point of $\Tilde{D}_\infty$, we could show that its limit gives rise to a positive Jacobi field. This would imply that $\Tilde{D}_\infty$ is stable. However, the only type of stable minimal surfaces in $\mathbb{R}^3$ is the planes, which contradicts the earlier conclusion that $\Tilde{D}_\infty$ is non-flat. 

    As a result, $\Tilde{D}_\infty$ is a properly embedded simply-connected non-flat minimal surface in $\mathbb{R}^3$. By the uniqueness of helicoid (Theorem 0.1 of \cite{Meeks-Rosenberg}), $\Tilde{D}_\infty$ must be a helicoid. For a detailed proof, see Section 2.1 of \cite{Meeks-Tinaglia_curv_est} or \cite{Meeks-Tinaglia_limit_lamination}.

    As argued in Section 2.2 and 2.3 in \cite{Meeks-Tinaglia_curv_est}, on the scale of the norm of the second fundamental form near the origin, the small limiting helicoid can be extended to multi-valued graphs contained in $D_n$, with mean curvature equal to one.
    Here a \emph{$N$-valued graph} is simply-connected $N$-sheeted covering of an annulus. In particular, half of a helicoid is an infinite multi-valued graph. For a detailed definition, see Definition 2.4 in \cite{Meeks-Tinaglia_curv_est}.
    
    More specifically, by passing to a subsequence and assuming that $n$ is sufficiently large, $D_n$ can contain an arbitrarily large number of disjoint multi-valued graphs. Consequently, as $n$ goes to infinity, at least two of these multi-valued graphs in $D_n$ become arbitrarily close to one another, with their mean curvature vectors pointing in opposite directions.
    However, since the mean curvature of $D_n$ is strictly positive, the limit of these graphs collapses, resulting in a configuration where the mean curvature vector points both upwards and downwards simultaneously. This is a contradiction, which disproves the assumption in \eqref{equ_assump_extrinsic}.
\end{proof}

As a corollary, the above lemma establishes a bound on the radii of $H$-disks for $H > 1$. However, note that when $H = 1$, the radii of embedded $1$-disks in $\mathbb{H}^3$ may be unbounded. For instance, the radius of a horosphere is infinite.

\begin{Cor}[Extrinsic radius estimates]\label{cor_extrinsic radius}
   There exists a constant $R>0$ such that the following estimate holds for any $H>1$. Let $D$ be an $H$-disk embedded in $\mathbb{H}^3$, then we have \begin{equation*}
        \sup_{x\in D}d_{\mathbb{H}^3}(x,\partial D)\leq R\tanh^{-1}\left(\frac{1}{H}\right).
    \end{equation*}
\end{Cor}

\begin{proof}
    Assuming by contradiction that the result is false, we can find a sequence of $H$-disks $D_n\subset \mathbb{H}^3$ whose extrinsic radii go to infinity. Let $\epsilon>0$ be a sufficiently small number, and let \begin{equation*}
        E_n:=\left\{x\in D_n: d_{\mathbb{H}^3}(x,\partial D_n)\geq\epsilon \tanh^{-1}\left(\frac{1}{H}\right)\right\}.
    \end{equation*}
    From Lemma \ref{lemma_weak extrinsic}, $E_n$ has uniformly bounded norm of the second fundamental form, and thus, the standard compactness argument indicates that $E_n$ converges smoothly on compact sets to a CMC-lamination $\mathcal{L}$ in $\mathbb{H}^3$. This lamination consists of surfaces with constant mean curvature $H$ and genus zero, which are almost embedded, meaning they do not intersect transversely. Let $E_\infty\subset \mathcal{L}$ be an arbitrary leaf.

    If $E_n$ converges to $E_\infty$ with multiplicity greater than two, then as shown in the previous lemma, $E_\infty$ is stable. Because $H>1$, the universal cover of $E_\infty$ is naturally associated with an $(H-1)$-surface in $\mathbb{R}^3$,  denoted by $\bar{E}_\infty$, where $H-1>0$. 
    The stability operator of $E_\infty$ and $\bar{E}_\infty$ are denoted by $L_{E_\infty}$ and $L_{\bar{E}_\infty}$. Let $\nu$ and $\bar{\nu}$ be unit normal vector fields for $E_\infty$ and $\bar{E}_\infty$, respectively. We have
    \begin{align*}
        L_{\bar{E}_\infty}=& \Delta_{\bar{E}_\infty}+|A_{\bar{E}_\infty}|^2+Ric^{\mathbb{R}^3}(\bar{\nu},\bar{\nu})\\
        =&\Delta_{E_\infty}+|A_{E_\infty}|^2-4H+2\\
        <& \Delta_{E_\infty}+|A_{E_\infty}|^2+Ric^{\mathbb{H}^3}(\nu,\nu)=L_{E_\infty}.
    \end{align*}
    Hence, $\bar{E}_\infty$ is also stable. 
    
    Complete stable surfaces immersed in $\mathbb{R}^3$ with nonzero constant mean curvature must be spheres, as shown in \cite{Barbosa-doCarmo}. Therefore, $\bar{E}_\infty$ is contained in a sphere of radius $\frac{1}{H-1}$, and then the universal cover of $E_\infty$ lies in a sphere of radius $\tanh^{-1}\left(\frac{1}{H}\right)$ in $\mathbb{H}^3$. As a consequence, when $n$ is sufficiently large,  $E_n$ is also contained in a sphere of radius $\tanh^{-1}\left(\frac{1}{H}\right)$. This implies that the diameter of $E_n$ is uniformly bounded, which contradicts the assumption that the radius of $D_n$ is unbounded. We conclude that the multiplicity is one or two.

Next, we demonstrate that the limit surface $E_\infty$ is proper. Assume that $E_\infty$ is not proper. In that case, for small $\delta>0$, the complement of $E_\infty$ in its closure would contain infinitely many components within an extrinsic ball of radius $\delta$. If $E_\infty$ were unstable, then there would be a compact, simply-connected subset $K$ in the universal cover of $E_\infty$ that is also unstable. Let $f$ denote the positive eigenfunction of the stability operator on $K$ associated with the smallest eigenvalue $\lambda_1<0$. For $-\delta\leq t\leq \delta$, consider the family of normal graphs $K(t):=\exp(tf\nu)$ over $K$, where $\nu$ is the unit normal vector field to $K$. The surfaces $K(t)$ have constant mean curvature greater than $H$ when $t>0$, and less than $H$ when $t<0$. Because of the maximum principle, the components of the limit leaf, all of which have mean curvature equal to $H$, cannot intersect the family $K(t)$. However, this leads to a contradiction, as such a configuration would violate the assumption of improperness. It means that $E_\infty$ must be stable. However, this is impossible as explained in the previous paragraph. So $E_\infty$ is proper. For more details, we refer readers to Theorem 6.1 in \cite{Meeks-Rosenberg_properness}.

Moreover, if $E_\infty$ has only one end, then it must be a horosphere, which is excluded by assuming $H>1$. So it has at least two ends. From the same proof as in Assertion 4.5 of \cite{Meeks-Tinaglia_Delaunay}, there is at least one annular end. According to \cite{Korevaar-Kusner-Meeks-Solomon}, each annular end of a proper $H$-surface in $\mathbb{H}^3$ with finite genus and bounded norm of the second fundamental form is asymptotically a \emph{Delaunay surface} $F_\infty$, namely a surface obtained by rotating the roulette of the conic. 

We claim that $E_n$ cannot converge to $F_\infty$ with finite multiplicity, leading to a contradiction. 
Assume, for the sake of contradiction, that $E_n$ converges to $F_\infty$ with finite multiplicity. Let $\gamma_n$ and $\gamma_\infty$ be closed geodesics in $E_n$ and $F_\infty$, respectively, such that $\gamma_n\rightarrow \gamma_\infty$. Due to the symmetry of $F_\infty$, it follows that $\gamma_\infty$ is contained in a ball of radius $r$ with $\frac{1}{\tanh r}=H$. Without loss of generality, we assume that the ball is centered at the origin, denoted by $B(0,r)$. Therefore, when $n$ is sufficiently large, $\gamma_n$ is also contained in a ball of the same radius. Let $G_n$ be the disk in $E_n$ that is bounded by $\gamma_n$. Suppose that $x_n\in G_n$ is the point that has the farthest distance to the boundary. Let $\nu_n$ be the normal vector of $G_n$ at $x_n$ pointing towards the origin, and let $\Pi_n^t$ be the plane perpendicular to $\nu_n$ at distance $t$ to the origin. By Alexandrov reflection principle, the connected component of $G_n\subset \Pi_n^{\frac{|x_n|+r}{2}}$ containing $x_n$ is a graph over $\Pi_n^{\frac{|x_n|+r}{2}}$. By maximum principle, the distance from $x_n$ to $\Pi_n^{\frac{|x_n|+r}{2}}$, that is $\frac{|x_n|-r}{2}$, is less than or equal to the radius of an $H$-sphere, that is $r$. Thus, \begin{equation*}
    \frac{|x_n|-\tanh^{-1}\left(\frac{1}{H}\right)}{2}\leq \tanh^{-1}\left(\frac{1}{H}\right),
\end{equation*}
and therefore \begin{equation*}
    |x_n|\leq 3r=3\tanh^{-1}\left(\frac{1}{H}\right).
\end{equation*}
As a consequence, when $n$ is large enough, each $G_n$ is contained in a ball of finite radius, thus it cannot converge to $F_\infty$ at infinity.    
\end{proof}

Next, we deduce a stronger extrinsic curvature estimate. 

\begin{Prop}[Extrinsic curvature estimates]\label{prop_extrinsic curv} 
Given $\delta>0$. There exists a constant $C=C(\delta)$ such that for any $H\geq 1$ and any embedded $H$-disk $D\subset\mathbb{H}^3$,
   we have \begin{equation*}
       \sup_{x\in D: \,d_{\mathbb{H}^3}(x,\partial D)\geq \delta}|A|\leq C.
   \end{equation*}
\end{Prop}

\begin{proof}
First, we use Lemma~\ref{lemma_weak extrinsic} and Corollary~\ref{cor_extrinsic radius} to discuss the case where the mean curvature of $D$ stays away from one and can possibly approach infinity. More specifically, we show the following claim.
\begin{Claim}
    Given $\delta>0$ and $h_0>1$, there exists a constant $C(\delta,h_0)$ such that the following holds. Let $H\geq h_0$ and let $D$ be an $H$-disk embedded in $\mathbb{H}^3$, we have \begin{equation*}
       \sup_{x\in D: \,d_{\mathbb{H}^3}(x,\partial D)\geq \delta}|A|\leq C(\delta,h_0).
   \end{equation*}
\end{Claim}

\begin{proof}[Proof of the claim]
     Arguing by contradiction, the result fails for some $\delta>0$ and $h_0>1$. There exist a sequence of $H_n$-disks $D_n$ in $\mathbb{H}^3$ with $H_n\geq h_0$, and points $x_n\in D_n$ with $d_{\mathbb{H}^3}(x_n,\partial D_n)\geq \delta$, such that \begin{equation}\label{equ_assump_ext_2}
        |A_{D_n}|(x_n)>n.
    \end{equation}
    Rescale the disk by $\Tilde{D}_n:=H_n(x_n)(D_n-x_n)$, then $\Tilde{D}_n$ is a $1$-disk contained in $\mathbb{H}^3\left(-\frac{1}{H_n^2}\right)$. We obtain the following inequalities. \begin{align}\label{equ_1}
         d_{\mathbb{H}^3\left(-\frac{1}{H_n^2}\right)}(0,\partial\Tilde{D}_n) &\geq \delta H_n,\\\nonumber
         |A_{\Tilde{D}_n}|(0)&>\frac{n}{H_n}.
    \end{align}
Let $\epsilon=\frac{\delta}{\tanh^{-1}\left(\frac{1}{h_0}\right)}$. Since the function \begin{equation*}
    x\tanh^{-1}\left(\frac{1}{x}\right)=\frac{x}{2}\ln\left(\frac{x+1}{x-1}\right)
\end{equation*}
decreases on $(1,\infty)$,
the former inequality of \eqref{equ_1} implies that \begin{equation*}
    d_{\mathbb{H}^3\left(-\frac{1}{H_n^2}\right)}(0,\partial\Tilde{D}_n)\geq \delta h_0=\epsilon h_0\tanh^{-1}\left(\frac{1}{h_0}\right)\geq \epsilon H_n\tanh^{-1}\left(\frac{1}{H_n}\right).
\end{equation*}
When combined with Lemma~\ref{lemma_weak extrinsic}, it gives rise to a constant $C'$ that depends only on $\epsilon$, thus depending on $\delta$ and $h_0$, such that 
\begin{equation*}
    |A_{\Tilde{D}_n}|(0)\leq C'.
\end{equation*}
Moreover, Corollary~\ref{cor_extrinsic radius} verifies the existence of $R>0$, so that for the $H_n$-disk $D_n\subset \mathbb{H}^3$, \begin{equation*}
    d_{\mathbb{H}^3}(0,\partial D_n)\leq R \tanh^{-1}\left(\frac{1}{H_n}\right).
\end{equation*}
After rescaling, 
\begin{equation}\label{equ_2}
    d_{\mathbb{H}^3\left(-\frac{1}{H_n^2}\right)}(0,\partial \Tilde{D}_n)\leq RH_n \tanh^{-1}\left(\frac{1}{H_n}\right).
\end{equation}
\eqref{equ_1}, together with \eqref{equ_2}, implies that \begin{equation*}
    \delta\leq R\tanh^{-1}\left(\frac{1}{H_n}\right)\quad \Longrightarrow\quad H_n\leq \frac{1}{\tanh\left(\frac{\delta}{R}\right)}.
\end{equation*}
Consequently, \begin{equation*}
    n\tanh\left(\frac{\delta}{R}\right)\leq \frac{n}{H_n}<|A_{\Tilde{D}_n}|(0)\leq C'.
\end{equation*}
However, since $n\tanh\left(\frac{\delta}{R}\right)$ diverges as $n$ goes to infinity, this contradicts the assumption \eqref{equ_assump_ext_2}. Thus, the proof of the claim is complete.
\end{proof}

To complete the proof of Proposition~\ref{prop_extrinsic curv}, we need to rule out the following scenario. Suppose that $D_n$ is a sequence of $H_n$-disks, where the constants $H_n\geq 1$ converge to one as $n\rightarrow\infty$. Assume there are points $x_n\in D_n$ such that $d_{\mathbb{H}^3}(x_n,\partial D_n)\geq \delta$ and $|A_{D_n}|(x_n)>n$. Consider $\Tilde{D}_n:=|A_{D_n}|(x_n)(D_n-x_n)$. The mean curvature of $\Tilde{D}_n$, given by $\frac{H_n}{|A_{D_n}|(x_n)}$, and the sectional curvature of the ambient space, given by $-\frac{1}{|A_{D_n}|^2(x_n)}$, both converge to zero as $n$ goes to infinity. Consequently, the same argument as in Lemma~\ref{lemma_weak extrinsic} that $\Tilde{D}_n$ converges smoothly on compact subsets to a minimal helicoid in $\mathbb{R}^3$, which leads to a contradiction.
\end{proof}

\subsection{Intrinsic curvature estimates}\label{subsection_intrinsic}
The remainder of the proof for the intrinsic curvature estimate follows the same procedure as the case in $\mathbb{R}^3$. For this reason, we only list the key lemmas and refer readers to Section 5 of \cite{Tinaglia_review} for the detailed arguments.

\begin{Lemma}[Weak chord arc property]\label{lem_chord arc}
    There exists $\epsilon>0$, such that for any $H$-disk $D\subset \mathbb{H}^3$ with $H>0$, and for any intrinsic closed ball $\bar{B}_D(x,R)\subset \mathrm{int} D$, the following statements hold. \begin{itemize}
        \item The extrinsic disk $D(x,\epsilon R):=D\cap B(x,\epsilon R)$ has piecewise smooth boundary $\partial D(x,\epsilon R)\subset \partial B(x,\epsilon R)$.
        \item $D(x,\epsilon R)\subset B_D(x,R)$.
    \end{itemize}
\end{Lemma}
The weak chord arc property for $H$-disks in $\mathbb{R}^3$ with $H>0$ is discussed by Meeks-Tinaglia in \cite{Meeks-Tinaglia_chord_arc}, which is an extension of the result of Colding-Minicozzi \cite{Colding-Minicozzi_chord} for the case of $H=0$.

The key step of the proof of above lemma is the one-sided extrinsic curvature estimate, which naturally follows from Proposition~\ref{prop_extrinsic curv}.

\begin{Lemma}[One-sided curvature estimates]
    There exists $C,\epsilon>0$, such that the following holds. Let $D$ be an $H$-disk embedded in $\mathbb{H}^3$ with $H>0$. In the Poincar{\'e} ball model, if $D\cap B\left(0,\frac{1}{2}\right)\cap \{z=0\}=\emptyset$, and $\partial D\cap B\left(0,\frac{1}{2}\right)\cap \{z>0\}=\emptyset$, then \begin{equation*}
        \sup_{D\cap B(0,\epsilon)\cap \{z>0\}} |A|\leq C.
    \end{equation*}
\end{Lemma}

Finally, the following weak intrinsic curvature estimate follows from the extrinsic curvature estimate (Proposition~\ref{prop_extrinsic curv}) and the weak chord arc property (Lemma~\ref{lem_chord arc}). 

\begin{Lemma}[Weak intrinsic curvature estimates]\label{lem_weak_intrinsic}
    Let $k>0$ and $H>k$. Given $\delta>0$, there exists a constant $C=C(\delta)$ depending only on $\delta$, such that for any $H$-disk $D$ embedded in $\mathbb{H}^3(-k^2)$, the simply-connected space with constant sectional curvature $-k^2$, we have \begin{equation*}
        \sup_{x\in D:\, d_D(x,\partial D)\geq \frac{\delta}{k}\tanh^{-1}\left(\frac{k}{H}\right)}|A|\leq CH.
    \end{equation*}
\end{Lemma}

Building on the proof of Corollary~\ref{cor_extrinsic radius}, we can derive the intrinsic radius estimate for $H$-disks with $H \geq h_0 > 1$ by applying Lemma~\ref{lem_weak_intrinsic}. Furthermore, by combining these results and revisiting the proof of Proposition~\ref{prop_extrinsic curv}, we establish the intrinsic curvature estimate for all $H\geq 1$ (Proposition~\ref{prop_intrinsic}).

\begin{Rem}
   The intrinsic curvature estimate for $H$-surfaces in $\mathbb{R}^3$ with $H>0$ was discussed in a sequence of papers by Meeks and Tinaglia \cite{Meeks-Tinaglia_Delaunay}\cite{Meeks-Tinaglia_chord_arc}\cite{Meeks-Tinaglia_curv_est}\cite{Meeks-Tinaglia_limit_lamination}. 

   When $H=0$, Colding and Minicozzi verified the weak chord arc property \cite{Colding-Minicozzi_chord} and one-sided curvature estimate \cite{Colding-Minicozzi_1-sided}. However, since the plane and helicoid are simply-connected minimal surfaces properly embedded in $\mathbb{R}^3$, the radius estimate is no longer true. Additionally, by rescaling a helicoid, we can construct a sequence of embedded minimal disks in $\mathbb{R}^3$ with the arbitrarily large norm of the second fundamental forms, which gives rise to a counterexample of Lemma~\ref{lemma_weak extrinsic}. Therefore, 
   the weak intrinsic or extrinsic estimate also fails for the minimal case.
\end{Rem}

\section{Area bound and compactness theorem}\label{section_area_compactness}
In this section, we provide the proof of Theorem~\ref{thm_area_bound} and Theorem~\ref{thm_compactness}. 
\subsection{Proof of Theorem~\ref{thm_area_bound}}
    \subsubsection{Define the singular set of convergence}
    
         Suppose by contradiction that $S_n$ is a sequence of $H_n$-surfaces embedded in $M$ with genus at most $g_0$ and $1\leq H_n\leq H_0$. Additionally, \begin{equation}\label{area>n}
            \area(S_n)>n.
        \end{equation}
        We observe that, after passing to a subsequence, the injectivity radius of $S_n$ goes to zero as $n$ goes to infinity. Because otherwise, the injectivity radius of $S_n$ is at least a positive constant $\delta>0$. 
       From Proposition~\ref{prop_curv_est_sep}, there exist constants $\epsilon$ and $k$ depending on $\delta$, and each $S_n$ has a one-sided $\epsilon$-neighborhood in $M$ such that \begin{equation*}
           \area(S_n)\leq k\, \text{vol}(N_\epsilon(S_n))\leq k\,\text{vol}(M)<\infty,
       \end{equation*}
       violating the assumption \eqref{area>n}. 

       From the discussion above, let $p_{1,n}\in S_n$ be the point that achieves the minimum value of the injectivity radius of $S_n$, then $I_{S_n}(p_{1,n})\rightarrow 0$. After passing to a subsequence, we assume that $p_{1,n}$ converges to $q_1\in M$. Furthermore, if $S_n\setminus\{p_{1,n}\}$ still does not have locally bounded injectivity radius, we choose the point $p_{2,n}$ that attains the minimum value on $S_n\setminus\{p_{1,n}\}$ and assume $p_{2,n}\rightarrow q_2\in M$. Repeating the procedure inductively, we find a countable set $\mathcal{Q}_0=\{q_1,q_2,\cdots\}$ of $M$, and let $\mathcal{Q}$ be the closure of $\mathcal{Q}_0$. Notice that, because of Proposition~\ref{prop_curv_est_sep}, $S_n$ has locally bounded norm of the second fundamental form in $M\setminus\mathcal{Q}$. The set $\mathcal{Q}$ is called the \emph{singular set of convergence}, because we will prove the convergence of $S_n$ away from $\mathcal{Q}$ later.

    \subsubsection{Discuss the classification of the singularities}
     Let $\delta_n$ be a sequence that converges to zero as $n\rightarrow \infty$.
     For any $q\in \mathcal{Q}$, analogous to the argument in Section 5 of \cite{Meeks-Tinaglia}, we can find a sequence $p_n\in S_n$ converging to $q$ such that the injectivity radius $I_{S_n}(p_n)<\frac{1}{n\,\delta_n}$. The continuous function \begin{equation*}
         h_n(x):=\frac{d_M\big(x,\partial B_M(p_n,\delta_n)\big)}{I_{S_n}(x)}
     \end{equation*}
     defined on the compact set $S_n\cap\overline{B_M(p_n,\delta_n)}$ has a maximum at an interior point $p_n'\in S_n\cap B_M(p_n,\delta_n)$. Let \begin{equation}\label{equ_sigma_n}
         \sigma_n:=d_M\big(p_n',\partial B_M(p_n,\delta_n)\big)>0,
     \end{equation}
     we have $\sigma_n\rightarrow 0$ as $n\rightarrow\infty$. 
     Moreover,
     \begin{equation}\label{equ_factor}
         \frac{\sigma_n}{I_{S_n}(p_n')}=h_n(p_n')\geq h_n(p_n)=\frac{\delta_n}{I_{S_n}(p_n)}>n.
     \end{equation}
     For any $x_n\in S_n\cap\overline{B_M(p_n',\frac{\sigma_n}{2})}$, we have \begin{equation*}
         \frac{\frac{\sigma_n}{2}}{I_{S_n}(x_n)}\leq h_n(x_n)\leq h_n(p_n')=\frac{\sigma_n}{I_{S_n}(p_n')}.
     \end{equation*}
     This implies that, for any point $x_n\in S_n\cap\overline{B_M(p_n',\frac{\sigma_n}{2})}$, the injectivity radius is uniformly bounded below, that is, 
     \begin{equation}\label{equ_injrad}
         I_{S_n}(x_n)\geq \frac{I_{S_n}(p_n')}{2}.
     \end{equation}
    Consider $B_{\Tilde{S}_n}(p_n',r)$ as the ball in the exponential coordinates $B_{T_{p_n'}\Tilde{S}_n}(0,r)$ centered at the origin, and let \begin{equation*}
        \lambda_n=\frac{1}{I_{S_n}(p_n')}\quad \text{and}\quad \Tilde{S}_n=\lambda_n\left(S_n\cap\overline{B_M\left(p_n',\frac{\sigma_n}{2}\right)}\right)\subset \overline{B_{T_{p_n'}\Tilde{S}_n}\left(0,\frac{\sigma_n}{2I_{S_n}(p_n')}\right)}.
    \end{equation*}
    $\Tilde{S}_n$ is a CMC surface in $\mathbb{H}^3\left(-\frac{1}{\lambda_n^2}\right)$.
     After rescaling, both the mean curvature $H(\Tilde{S}_n)=\frac{H_n}{\lambda_n}$ of $\Tilde{S}_n$ and its ambient sectional curvature $-\frac{1}{\lambda_n^2}$ tend to zero as $n\to \infty$.
    Notice that the injectivity radius satisfies $I_{\Tilde{S}_n}(0)=1$, and by \eqref{equ_injrad}, for any $x_n\in \Tilde{S}_n$, we have $I_{\Tilde{S}_n}(x_n)\geq \frac{1}{2}$.
    
    However, we observe that the norm squared second fundamental form of $\Tilde{S}_n$ may not be locally bounded. Specifically, because the mean curvature of the rescaled surface $\Tilde{S}_n$ converges to zero, the contradiction argument based on the maximum principle, used in the proof of the weak extrinsic estimate (Lemma~\ref{lemma_weak extrinsic}), fails. Moreover, according to \eqref{equ_factor}, the radius $\lambda_n\frac{\sigma_n}{2}$ of $\Tilde{S}_n$ diverges as $n\rightarrow\infty$, so the radius estimate (Corollary~\ref{cor_extrinsic radius}) also breaks down. As a result, the strategy for establishing the curvature estimate in Proposition~\ref{prop_curv_est_sep} does not work. Consequently, in this scenario, having a positive lower bound on the injectivity radius is not equivalent to having a local upper bound on the second fundamental form. This leads us to consider the following two cases.
    \begin{itemize}
        \item If $\Tilde{S}_n$ has locally bounded norm of the second fundamental form, then it converges to a minimal surface $\Tilde{S}\subset \mathbb{R}^3$ smoothly on compact sets away from $\mathcal{Q}$. 
        Besides, if the multiplicity is at least three, then the same reasoning as in Lemma~\ref{lemma_weak extrinsic} demonstrates that $\Tilde{S}$ is stable. 
         The only stable minimal surfaces in $\mathbb{R}^3$ are the planes,  
         which have infinite injectivity radii. However, since $I_{\Tilde{S}}(0) = 1$, $\Tilde{S}$ does not contain any planes. This implies that its multiplicity must be one or two. Additionally, the genus of $\Tilde{S}$ is at most $g_0$. If $\Tilde{S}$ has genus zero and a single end, then, since it is not a plane, it must be a helicoid. However, this would violate the injectivity radius condition at the origin. Therefore, if $\Tilde{S}$ has genus zero, it must have more than one end, meaning it is either a catenoid with finite topology or a Riemann minimal example with infinite topology (\cite{Meeks-Tinaglia}, Section 5). The latter possibility can be ruled out below.

     If $\Tilde{S}$ were a Riemann minimal example, $\partial\Tilde{S}$ would contain a large number of pairwisely disjoint simple closed curves that are homotopically nontrivial. By applying Lemma 5.4 of \cite{Meeks-Tinaglia}, we find that at least two of these curves bound an annulus in $\Tilde{S}$. Due to the convergence and the dilation, for sufficiently large $n$, there would be corresponding homotopically nontrivial curves in $S_n$, and at least two of them would bound an annulus in $S_n$. However, this contradicts Claim 5.3 of \cite{Meeks-Tinaglia}. (Although Claim 5.3 specifically addresses constant mean curvature surfaces in a flat 3-torus, its proof also applies to our case).

        \item Otherwise, the norm of the second fundamental form of $\Tilde{S}_n$ is not locally bounded.
        
        If $H(S_n)=H_n\rightarrow 1$ as $n\rightarrow\infty$,
        then for sufficiently large $n$, 
        \begin{equation*}
             \int_{\Tilde{S}_n}|\mathring{A}|^2\,dA\leq \int_{S_n}|\mathring{A}|^2\,dA \approx\int_{S_n}\left(|\mathring{A}|^2-2(H(S_n)^2-1)\right)dA,
        \end{equation*}
        where $\mathring{A}=A-HI_2$ is the traceless part of the second fundamental form, and $I_2$ is the identity matrix.
        It follows from the Gauss-Bonnet formula that
        \begin{align}\label{equ_Gauss of Bryant}
             \int_{\Tilde{S}_n}|\mathring{A}|^2\,dA &\leq \int_{S_n}\left(|\mathring{A}|^2-2(H(S_n)^2-1)\right)dA+1\\\nonumber
            &=\int_{S_n}-2K\,dA+1=8\pi(g(S_n)-1)+1\\\nonumber
            &\leq 8\pi(g_0-1)+1.
        \end{align}

        By replacing $A$ with $\mathring{A}$ and applying the generalized Simons' inequality along with the generalized mean value inequality, the Choi-Schoen type small curvature estimate for minimal surfaces in $\mathbb{R}^3$ (see \cite{Choi-Schoen}) can be extended to CMC surfaces in $\mathbb{H}^3$. More specifically, the following result can be directly deduced from Theorem 4.1 in \cite{Zhang}, which generalizes the Choi-Schoen curvature estimate to the traceless second fundamental form of CMC surfaces in 3-manifolds. Alternatively, since a CMC surface in $\mathbb{H}^3$ with constant mean curvature slightly greater than and close to one corresponds to a CMC surface in $\mathbb{R}^3$ with sufficiently small constant mean curvature, we can apply Theorem 1.1 of \cite{Bourna-Tinaglia} or Theorem 2.4 of \cite{Sun}, which provide estimates for $|A|$ for surfaces in $\mathbb{R}^3$.
        
        As a result, suppose $p\in S_n$ and $B_M(p,r)\cap \partial S_n=\emptyset$. There exists a constant $\eta$ not depending on $n$, such that if \begin{equation*}
            \int_{S_n\cap B_M(p,r)}|\mathring{A}|^2\,dA\leq \eta,
        \end{equation*}
        then we have \begin{equation*}
            \max_{0\leq s\leq r}s^2\sup_{S_n\cap B_M(p,r-s)}|\mathring{A}|^2\leq C(p,r).
        \end{equation*}
        Recall that the sequence $\sigma_n$, as defined in \eqref{equ_sigma_n}, approaches zero as $n\rightarrow\infty$. By substituting $\delta_n$ in the expression of $\sigma_n$ with a smaller constant, if needed, we obtain
        \begin{equation}\label{equ_eta}
            \int_{\Tilde{S}_n}|\mathring{A}|^2\,dA=\int_{S_n\cap \overline{B_M(p_n',\frac{\sigma_n}{2})}}|\mathring{A}|^2\,dA \leq \eta
        \end{equation} 
        whenever $n$ is sufficiently large. 
        Then, the norm of the second fundamental form of $\Tilde{S}_n$, that is \begin{equation*}
            |A_{\Tilde{S}_n}|^2=|\mathring{A}_{\Tilde{S}_n}|^2+2H(\Tilde{S}_n)^2\rightarrow |\mathring{A}_{\Tilde{S}_n}|^2\quad \text{as }n\rightarrow \infty
        \end{equation*} is locally bounded, contradicting the assumption. 
        
        As a result, after extracting a subsequence, we can assume that $H(S_n)\geq h_0$ for some $h_0>1$ while $n$ is large enough.
        Let $f_n$ and $g_n$ be the meromorphic functions in the Weierstrass representation of the CMC surfaces $\Tilde{S}_n$ in $\mathbb{H}^3\left(-\frac{1}{\lambda_n^2}\right)$, then the same $f_n$ and $g_n$ determine a surface $\Tilde{\Sigma}_n$ in $\mathbb{R}^3$ with positive constant mean curvature homeomorphic to $\Tilde{S}_n$. Therefore, $\Tilde{\Sigma}_n$ has locally unbounded norm of the second fundamental form, and from Theorem 1.5 of \cite{Meeks-Tinaglia_classification}, the sequence converges smoothly on compact sets to a minimal parking garage structure of $\mathbb{R}^3$ with two oppositely oriented columns, denoted by $\Tilde{S}$. 
        Notice that $\Tilde{S}$ is also the limit of $\Tilde{S}_n$. However, this is not possible for the same reason as in the case of the Riemann minimal example.      
    \end{itemize}
    In summary, $\Tilde{S}_n$ has locally bounded norm of the second fundamental form, and it is of one of the following types. \begin{enumerate}
        \item $\Tilde{S}$ is either a catenoid,
        \item or it is an embedded minimal surface with positive genus at most $g_0$.
    \end{enumerate}

     \subsubsection{Prove the finiteness of $\mathcal{Q}$}\label{subsection_finiteness of singularities}
     
     From the previous discussion, for any singular point of convergence $q\in\mathcal{Q}$, we can find a sequence $\Tilde{S}_n$ by zooming in the singularity so that $\Tilde{S}_n$ converges to a minimal surface in $\mathbb{R}^3$ of type (1) or (2). Hence we say that $q$ is either a catenoid singular point or a non-catenoid singular point of genus $1\leq g\leq g_0$.
     
     We claim that $\mathcal{Q}$ is finite. If $H(S_n)\rightarrow 1$ as $n\rightarrow\infty$, then based on \eqref{equ_Gauss of Bryant} and $\eta$ in \eqref{equ_eta}, there exist at most finitely many $p_n\in S_n$, such that 
     \begin{equation*}
         \int_{S_n\cap B(p_n,r)}|\mathring{A}|^2\,dA>\eta.
     \end{equation*}
     The Choi-Schoen type estimate implies that, because the second fundamental form becomes unbounded near the singular points, each point in $\mathcal{Q}$ must be the limit of a sequence $\{p_n\}_{n\in\mathbb{N}}$. Consequently, the finiteness of the sequence $\{p_n\}_{n\in\mathbb{N}}$ indicates that $\mathcal{Q}$ is also finite.
     
     Consider the general case where $1\leq H(S_n)\leq H_0$.
     Regarding singularities of type (2), the number of non-catenoid singular points is at most $g_0$. Let $\mathcal{Q}_c\subset\mathcal{Q}$ represent the set of catenoid singular points. 
     According to Sections 5-6 of \cite{Meeks-Tinaglia}, if there exists an integer $k\geq 1$, such that \begin{equation}\label{equ_number of singularities}
         \#\mathcal{Q}_c> k(3g_0-2),
     \end{equation} then we can find at least $k$ pairs of points in $\mathcal{Q}_c$, so that each of them, denoted by $\{p_i,q_i\}$, serves as the center of two curves that bound an annulus $A_n(p_i,q_i)$ in $S_n$. It corresponds to an annulus $\Tilde{A}_n(p_i,q_i)$ in a lift $\Tilde{S}_n$ to $\mathbb{H}^3$.
     
     For each pair, let $l_n(p_i,q_i)$ be the straight line passing through the center of $\Tilde{A}_n(p_i,q_i)$, around which the annulus is rotationally symmetric.
     We compare the mean curvature of $\Tilde{A}_n(p_i,q_i)$ with that of the cylinders whose central axis of rotational symmetry is $l_n(p_i,q_i)$. Denote such a cylinder of radius $r$ by $C_n(p_i,q_i,r)$. By the maximum principle, if $\Tilde{A}_n(p_i,q_i)$ is tangent from the outside to $C_n(p_i,q_i,r)$ at the farthest point of $\Tilde{A}_n(p_i,q_i)$ from $l_n(p_i,q_i)$, then \begin{equation*}
         H(C_n)=\frac{1}{2\tanh(r)}<H(S_n)\leq H_0
     \end{equation*}
     implies that 
     \begin{equation*}
         r>\tanh^{-1}\left(\frac{1}{2H(S_n)}\right)\geq \tanh^{-1}\left(\frac{1}{2H_0}\right).
     \end{equation*}
     As a result, the farthest distance of $\Tilde{A}_n(p_i,q_i)$ from its central axis of rotational symmetry is bounded below by a fixed positive constant $\tanh^{-1}\left(\frac{1}{2H_0}\right)$. 
     
     Using the Alexandrov reflection principle, there exists an open set $\Tilde{G}(p_i,q_i)\subset \mathbb{H}^3$ associated with $\Tilde{A}_n(p_i,q_i)$, along with a constant $\Tilde{v}>0$, such that $\Tilde{G}(p_i,q_i)$ lies entirely within the mean convex side of $\Tilde{S}_n$, and the volume of $\Tilde{G}(p_i,q_i)$ is at least $\Tilde{v}$. Notably, $\Tilde{v}$ does not depend on $n$ or $i$, but is instead determined by the uniform upper bound $H_0$ of the mean curvature of $S_n$. Due to the separation property of $S_n$, there exists a corresponding open set $G(p_i,q_i)\subset M$ and $v>0$, such that $G(p_i,q_i)$ is contained in the mean convex side of $S_n$ in $M$, and the volume of $G(p_i,q_i)$ is at least $v$. Moreover, for different points $p_j,q_j\in\mathcal{Q}_c$, $G(p_i,q_i)$ and $G(p_j,q_j)$ are disjoint. Thus, we have the volume estimate:
     \begin{equation}\label{equ_volume_lower}
         \text{vol}(M)\geq \sum_{i=1}^k\text{vol}(G(p_i,q_i))\geq kv.
     \end{equation}
     This implies that $k$ is finite. Using \eqref{equ_number of singularities}, we deduce \begin{equation*}
         \#\mathcal{Q}\leq\#\mathcal{Q}_c+g_0\leq (k+1)(3g_0-2)+g_0\leq \left(\frac{\text{vol}(M)}{v}+1\right)(3g_0-2)+g_0.
     \end{equation*}
     $\#\mathcal{Q}$ is therefore finite.

    \subsubsection{Deduce the convergence}
    
     From Proposition~\ref{prop_curv_est_sep}, the second fundamental forms of $S_n$ are locally bounded on compact sets of $M\setminus\mathcal{Q}$. Furthermore, combined with the local area estimate in Theorem 3.5 of \cite{Meeks-Tinaglia_epsilon-nbhd}, this implies the existence of a constant $C_0>0$, such that $\area\left(S_n\cap B_M\left(p,\frac{\epsilon}{2}\right)\right)<C_0$ for all $n\in\mathbb{N}$ and all $p\in M\setminus\mathcal{Q}$, where $\epsilon$ satisfies the condition that $B_M(p,\epsilon)$ does not intersect $\mathcal{Q}$. As a result, a standard compactness theorem applies, and, up to a subsequence, $S_n$ converges smoothly to a surface or lamination, denoted by $S$, on compact sets of $M\setminus\mathcal{Q}$. According to the proof of Corollary~\ref{cor_extrinsic radius}, $S$ is properly immersed in $M$. Moreover, $S$ has constant mean curvature $H$, where $1\leq H\leq H_0$. 

     We now prove that the convergence has multiplicity one. As established earlier, $S_n$ separates $M$ into two components, one of which is strictly mean convex, we denote it by $B_{S_n}$. For any $p\in S\setminus\mathcal{Q}$, $S\cap B_M(p,\epsilon)$ is disjoint from $\mathcal{Q}$. A connected component of $S\cap B_M(p,\epsilon)$, denoted by $S'(p)$, can be expressed as a graph over the tangent plane $T_pS$ of $S$ at $p$. $S'(p)$ is obtained as the limit of a sequence of graphs $S_n'(p)\subset S_n$, each defined over $T_pS$. If the convergence had multiplicity greater than one, ensuring that $S_n$ separates $M$ and bounds a strictly mean convex domain would create a contradiction. Specifically, the mean curvatures of two adjacent components of $S_n'(p)$ would have to point in opposite directions, which would make it impossible to define a consistent direction for the mean curvature vector on $S'(p)$. Thus, we conclude that $S_n\setminus\mathcal{Q}$ converges to $S$ with multiplicity one. Additionally, $B_{S_n}\setminus\mathcal{Q}$ converges to a strictly mean convex domain $B_S$ with $\partial B_S=S$. In other words, $S$ is strongly Alexandrov embedded in $M\setminus\mathcal{Q}$.

     \subsubsection{Find the contradiction}\label{subsection_contradiction}
     
     Since $S_n\setminus\mathcal{Q}$ converges to $S$ with multiplicity one, for any $\epsilon>0$, we can find $C(\epsilon)>0$ such that \begin{equation*}
         \area(S_n\setminus (\cup _{q\in\mathcal{Q}}B_M(q,\epsilon)))\rightarrow \area(S\setminus (\cup _{q\in\mathcal{Q}}B_M(q,\epsilon))) \leq C(\epsilon).
     \end{equation*}
     Thus when $n$ is sufficiently large, we have
     \begin{equation}\label{equ_non-singular}
         \area(S_n\setminus (\cup _{q\in\mathcal{Q}}B_M(q,\epsilon)))\leq C(\epsilon)+1.
     \end{equation}
     Now consider the neighborhood of singular points. Note that the points in $\mathcal{Q}$ are isolated due to finiteness. So we can fix an sufficiently small $\epsilon>0$, such that for distinct $p,q\in\mathcal{Q}$, $B_M(p, 2\epsilon)$ and $B_M(q,2\epsilon)$ are disjoint, and \begin{equation}\label{equ_e}
         4e^{-\epsilon H_0}\geq 2.
     \end{equation} According to the monotonicity formula (for example, see 40.2 of \cite{Simon}), for any $q\in\mathcal{Q}$, \begin{equation}\label{equ_monotonicity}
         e^{2\epsilon H_n}\frac{\area(S_n\cap B_M(q,2\epsilon))}{\omega_2\,4\epsilon^2}\geq e^{\epsilon H_n}\frac{\area(S_n\cap B_M(q,\epsilon))}{\omega_2\,\epsilon^2},
     \end{equation}
     where $\omega_2=4\pi \sinh^2(1)$ is the area of the unit disk in $\mathbb{H}^3$. When combined with \eqref{equ_e}, it implies that \begin{equation*}
         \area(S_n\cap B_M(q,2\epsilon))\geq 4e^{-\epsilon H_0}\area(S_n\cap B_M(q,\epsilon))\geq 2\,\area(S_n\cap B_M(q,\epsilon)).
     \end{equation*}
     Thus, when $n$ is large enough, \begin{align}\label{equ_singular}
         \area(S_n\cap B_M(q,\epsilon))&\leq \area(S_n\cap B_M(q,2\epsilon)) -\area(S_n\cap B_M(q,\epsilon))\\\nonumber
         &\leq \area(S_n\setminus (\cup_{q\in\mathcal{Q}} B_M(q,\epsilon)))\leq C(\epsilon)+1.
     \end{align}
     Combining \eqref{equ_non-singular} and \eqref{equ_singular}, we obtain that \begin{equation*}
         \area(S_n)\leq (\#\mathcal{Q}+1)(C(\epsilon)+1),
     \end{equation*}
     it is uniformly bounded, thus violating the assumption \eqref{area>n}. This finishes the proof of the theorem.

\subsection{Proof of Theorem~\ref{thm_compactness}}
 In accordance with the area bound in Theorem~\ref{thm_area_bound}, $S_n$ has uniformly bound total curvature, \begin{align*}
        \int_{S_n}|A|^2\,dA &= 8\pi(g(S_n)-1)+2\int_{S_n}(-1+2H^2)\,dA\\
        &\leq 8\pi(g_0-1)+2(2H_0^2-1)C(M,H_0,g_0).
    \end{align*}
    Hence, the total curvature of the limit $S$ is bounded by the same number. For any $\eta>0$, by choosing $\epsilon>0$ small enough, we have that for any $q\in \mathcal{Q}$, \begin{equation*}
        \int_{S\cap B(q,\epsilon)}|A|^2\,dA<\eta.
    \end{equation*}
    It follows from Theorem 4.3 of \cite{Sun} that $S$ extends smoothly across $\mathcal{Q}$, and $\mathcal{Q}$ is contained in the set of self-intersections of $S$. Moreover, in a small neighborhood of each non-embedded point $q\in\mathcal{Q}$, $S$ decomposes into two smoothly embedded disks touching at $q$, thus $S$ is almost embedded in $M$. 

    Finally, using the monotonicity formula \eqref{equ_monotonicity} and following the calculations in Section 9 of \cite{Meeks-Tinaglia}, we conclude that all points in $\mathcal{Q}$ are catenoid singular points. 

\begin{Rem}
If $S$ is a closed $H$-surface immersed in $M$ with $0\leq |H|\leq H_0< 1$ and genus at most $g_0$, then it follows immediately from the Gauss-Bonnet formula that \begin{equation*}
    \area(S)\leq \frac{1}{1-H^2}\left(4\pi(g_0-1)-\int_S\frac{|\mathring{A}|^2}{2}\,dA\right)\leq \frac{4\pi(g_0-1)}{1-H_0^2}.
\end{equation*}
The area of $S$ is also uniformly bounded by a constant depending only on $H_0$ and $g_0$.

However, even if $S$ is embedded, $S$ may not separate $M$ (see details in Lemma~\ref{lem_separating}). This prevents us from deriving the same compactness result as stated in Theorem~\ref{thm_compactness}. In particular, the convergence of a sequence of such surfaces may not occur with multiplicity one.
\end{Rem}

\section{Area bound for Bryant surfaces}\label{section_bryant}
In this section, we prove Theorem~\ref{thm_bryant area bound} and Corollary~\ref{cor_bryant area bound}. We start by introducing the following lemmas.

\begin{Lemma}\label{thm_genus>2}
Suppose $M$ is a closed hyperbolic 3-manifold.
    Let $S$ be a closed Bryant surface immersed in $M$, then the genus of $S$ is at least three.
\end{Lemma}

\begin{proof}

    The Gaussian curvature of $S$ satisfies \begin{equation}\label{equ_gaussian}
        K=-1+H^2-\frac{|\mathring{A}|^2}{2}=-\frac{|\mathring{A}|^2}{2}\leq 0,
    \end{equation} 
    and from the Gauss-Bonnet formula, the genus of $S$, denoted by $g$, is at least one. If $g=1$, then $S$ is a torus and the total Gaussian curvature over $S$ is zero, thus it vanishes at each point. Therefore, at each point of $S$, $\mathring{A}$ vanishes. The lift of $S$ to $M$ must be a horosphere, but this is impossible since the projection of any horosphere to a closed hyperbolic manifold is dense (see Ratner \cite{Ratner} or Shah \cite{Shah}).
    Next, assume that $g=2$. Bryant \cite{Bryant} proved that the hyperbolic Gauss map $\Tilde{G}: \Tilde{S}\rightarrow \mathbb{S}^2$ for the lifting Bryant surface $\Tilde{S}\subset\mathbb{H}^3$ is holomorphic, so the hyperbolic Gauss map $G:S\rightarrow \mathbb{S}^2$ is also holomorphic. We can consider $S$ as a conformal branch cover of $\mathbb{S}^2$ with degree $g-1=1$. However, the only one-sheeted cover of $\mathbb{S}^2$ is $\mathbb{S}^2$. Hence the $g$ is at least three. 
\end{proof}

\begin{Lemma}\label{thm_handlebody}
Suppose that $M$ is a finite-volume hyperbolic 3-manifold.
    Let $S$ be a closed Bryant surface embedded in $M$, and let $B(S)$ be the mean convex component of $M\setminus S$. Then $B(S)$ is an open handlebody, and the induced map $\pi_1(S)\rightarrow \pi_1(B(S))$ is surjective.
\end{Lemma}

\begin{proof}
By Lemma~\ref{lem_separating} that is stated later, $S$ separates $M$ into two components, and here $B(S)$ stands for the mean convex component. 
Suppose by contradiction that $\pi_1(S)\rightarrow \pi_1(B(S))$ is not surjective. Then the lift of $S$ to $\mathbb{H}^3$ consists of at least two disjoint Bryant surfaces, they are boundary components of a mean convex set in $\mathbb{H}^3$. But this violates the halfspace theorem in \cite{Rodriguez-Rosenberg}, which states that the connected set of $\mathbb{H}^3$ bounded by two disjoint embedded Bryant surfaces cannot be mean convex. Therefore, $\pi_1(S)\rightarrow \pi_1(B(S))$ is surjective. And since $\partial B(S)=S$, $B(S)$ is an open handlebody.
\end{proof}

\subsection{Proof of Theorem~\ref{thm_bryant area bound}}
Let $M$ be a closed hyperbolic $3$-manifold.
Suppose by contradiction that there exists a sequence of closed Bryant surfaces $S_n$ immersed in $M$ such that \begin{equation*}
    \frac{\area(S_n)}{g(S_n)}\rightarrow \infty\quad \text{as }n\rightarrow \infty,
\end{equation*}
where $g(S_n)$ is the genus of $S_n$. 
It follows from \eqref{equ_gaussian} that \begin{equation*}
    \fint_{S_n}\frac{|\mathring{A}|^2}{2}\,dA=-\fint_{S_n}K\,dA=\frac{4\pi(g(S_n)-1)}{\area(S_n)}\rightarrow 0.
\end{equation*}

\begin{Lemma}\label{lemma_total_curvature}
When $M$ is lifted to the universal cover $\mathbb{H}^3$, there exists a component of the associated lift of $S_n$ in $\mathbb{H}^3$, denoted by $\Tilde{S}_n$, such that for all $R>0$,
    \begin{equation*}
        \int_{\Tilde{S}_n\cap B(0,R)}|\mathring{A}|^2\,dA\rightarrow 0.
    \end{equation*}
\end{Lemma}

If the lemma holds, then for $R_n\rightarrow \infty$, $\sup_{\Tilde{S}_n\cap B(0,R_n)} |\mathring{A}|$ converges to zero. Thus, due to the standard compactness result, $\Tilde{S}_n\cap B(0,R_n)$ converges smoothly on compact sets to a Bryant surface $\Tilde{S}$ with $|\mathring{A}|=0$ for all points. Therefore, $\Tilde{S}$ must be a horosphere. Moreover, the lemma indicates that the total Gaussian curvature of $S_n$ converges to zero. Thus by Gauss-Bonnet formula, when $n$ is sufficiently large, $S_n$ must be a torus, but this violates Lemma~\ref{thm_genus>2}.

\begin{proof}[Proof of Lemma \ref{lemma_total_curvature}]
We now prove the lemma following the strategy by Calegari, Marques, and Neves in \cite[Section 6]{Calegari-Marques-Neves}.

For any $\phi\in \Gamma:=\pi_1(M)$, denote by $\phi(\Tilde{S}_n)$ a component of the corresponding lift of $S_n$ in $\mathbb{H}^3$. Let $\Delta\subset \mathbb{H}^3$ represent a fundamental domain for $M$ whose boundary is transverse to $\phi(\Tilde{S}_n)$ if their intersection is non-empty. $\phi(\Tilde{S}_n)\cap\Delta$ bounds a mean convex domain $\phi(\Tilde{B}(S_n))\cap\Delta$. Since  $\phi(\Tilde{B}(S_n))\cap\Delta$ is not necessarily simply connected, consider the universal cover of $\phi(\Tilde{B}(S_n))\cap\Delta$ and denote the fundamental domain for the action of $\pi_1(\phi(\Tilde{B}(S_n))\cap\Delta)$ by $\Delta_n$. Then $\partial\Delta_n$ is a lift of $\phi(\Tilde{S}_n)\cap\Delta$, and hence $\area(\phi(\Tilde{S}_n)\cap\Delta)=\area(\partial \Delta_n)$.
By Lemma~\ref{thm_handlebody}, $\partial\Delta_n$ has genus equal to zero and with at most $g(S_n)$ boundary components, we can use the method of Theorem~\ref{thm_area_bound} to show that the area of $\partial \Delta_n$ is bounded by a constant depending only on $M$.

More specifically, repeating the proofs of Theorem~\ref{thm_area_bound} and Theorem~\ref{thm_compactness}, we conclude that, after passing to a subsequence, $\partial \Delta_n$ converges smoothly to a Bryant surface, away from a singular set $\mathcal{Q}$ and the boundary components of $\partial\Delta_n$. It then follows from the previous argument that all points in $\mathcal{Q}$ are catenoid singular points. 
Moreover, because each hole of $S_n$ corresponds to a hole in the handlebody on the mean convex side, for any closed essential curve in the handlebody that bounds a disk on the complementary component of $M\setminus S_n$, it corresponds to an annular end in $\partial \Delta_n$. 
As shown in Section~\ref{subsection_finiteness of singularities}, particularly by \eqref{equ_volume_lower}, each annulus bounds an open region on the mean convex side with a volume that is at least a fixed amount. Since $\vol(\Delta_n)=\vol(\phi(\Tilde{B}(S_n))\cap\Delta)\leq \vol(\Delta)$, the total volume is finite, 
there can only be finitely many annuli. Consequently, both the cardinality of singular set and the number of boundary components of $\partial\Delta_n$ are uniformly bounded. 
Finally, as in the calculations in Section~\ref{subsection_contradiction}, we deduce the existence of a constant $C>0$ independent of $n$ such that \begin{equation*}
        \area(\phi(\Tilde{S}_n)\cap \Delta)=\area(\partial\Delta_n)\leq C,\quad \forall \phi\in \Gamma.
    \end{equation*}
    
Let $\Pi_n$ be the image of $\pi_1(S_n)$ in $\Gamma$, we have \begin{equation}\label{equ_cmn_area_bound}
        \area(S_n)=\sum_{\phi\in \Gamma/\Pi_n}\area(\phi(\Tilde{S}_n)\cap \Delta)=\sum_{\phi\in \Gamma^{S_n}}\area(\phi(\Tilde{S}_n)\cap \Delta)\leq C\#\Gamma^{S_n},
    \end{equation}
where \begin{equation*}
    \Gamma^{S_n}=\left\{\phi\in\Gamma/\Pi_n: \phi(\Tilde{S}_n)\cap \Delta\neq \emptyset\right\}
\end{equation*}
is a non-empty set.
For $\epsilon>0$, define \begin{equation*}
    \Gamma^{S_n}(\epsilon,R)=\left\{\phi\in \Gamma^{S_n}: \int_{\phi(\Tilde{S}_n)\cap B(0,R)}|\mathring{A}|^2\,dA\leq \epsilon\right\}.
\end{equation*}
To prove the lemma, we need to show that for any $R$ and any sufficiently small $\epsilon$, $\Gamma^{S_n}(\epsilon,R)\neq \emptyset$.
According to \cite[Lemma 2.2]{Calegari-Marques-Neves}, for each $R$, there exists $m_R>0$, such that $B(0,R)\subset \cup_{|\phi|\leq m_R}\phi(\Delta)$. Let $C_R$ be the number of elements $\phi\in \Gamma$ with $|\phi|\leq m_R$, it is non-zero. We have the following inequality. 
\begin{align*}
    &C_R\int_{S_n}|\mathring{A}|^2\,dA\\
    \geq & \sum_{|\phi|\leq m_R}\sum_{\phi\in \Gamma/ \Pi_n}\int_{\phi(\Tilde{S}_n)\cap \Delta}|\mathring{A}|^2\,dA
    = \sum_{\phi\in \Gamma/ \Pi_n} \sum_{|\phi|\leq m_R}\int_{\phi(\Tilde{S}_n)\cap \Delta}|\mathring{A}|^2\,dA\\
    \geq & \sum_{\phi\in \Gamma/ \Pi_n}\int_{\phi(\Tilde{S}_n)\cap B(0,R)}|\mathring{A}|^2\,dA
    \geq \sum_{\phi\in \Gamma^{S_n}-\Gamma^{S_n}(\epsilon,R)}\int_{\phi(\Tilde{S}_n)\cap B(0,R)}|\mathring{A}|^2\,dA\\
    \geq  & \epsilon\, \#(\Gamma^{S_n}-\Gamma^{S_n}(\epsilon,R)).
\end{align*}
Thus, \begin{equation*}
    \frac{\#(\Gamma^{S_n}-\Gamma^{S_n}(\epsilon,R))}{\area(S_n)}\leq \frac{C_R}{\epsilon}\fint_{S_n}|\mathring{A}|^2\,dA\rightarrow 0\quad \text{as }n\rightarrow \infty.
\end{equation*}
Combining it with \eqref{equ_cmn_area_bound}, we obtain \begin{equation*}
    \frac{\#(\Gamma^{S_n}-\Gamma^{S_n}(\epsilon,R))}{\#\Gamma^{S_n}}\rightarrow 0 \quad \text{as }n\rightarrow \infty.
\end{equation*}
As a result, $\Gamma^{S_n}(\epsilon,R)$ must be non-empty. This proves the lemma.
\end{proof}

\subsection{Proof of Corollary~\ref{cor_bryant area bound}}
Now consider a finite-volume hyperbolic $3$-manifold $M$. In this setting, Lemma~\ref{thm_genus>2} may fail if $M$ has cusps. The reason is that, in the proof of the lemma, a horosphere may project to a horotorus in $M$, which is a closed Bryant surface of genus one. However, if we assume that the genus $g$ of a Bryant surface is not equal to one, then the rest of the argument still applies and shows that $g\geq 3$.

Therefore, it remains to verify Lemma~\ref{lemma_total_curvature} in the finite-volume case. In \cite{Calegari-Marques-Neves}, it was shown that on a closed hyperbolic $3$-manifold, among all metrics with sectional curvature less than or equal to $-1$, the minimal surface entropy (the exponential asymptotic growth of the number of essential minimal surfaces up to a given value of area) is minimized if and only if the metric is hyperbolic. In the most recent work by the author and Vargas Pallete \cite[Theorem B]{Jiang-VargasPallete}, this result was extended to finite-volume hyperbolic $3$-manifolds, under the assumptions that the metrics are bilipschitz equivalent to the hyperbolic metric and admit a uniform lower bound on sectional curvature. In particular, the total curvature estimates for closed surfaces in closed hyperbolic $3$-manifolds (Lemma~\ref{lemma_total_curvature}) can be generalized to ambient manifolds with cusps, which provides a key ingredient for proving the ``only if'' direction. Consequently, Corollary~\ref{cor_bryant area bound} follows.

\bibliographystyle{plain} 
\bibliography{ref}   
\end{document}